\documentclass{amsart}
\usepackage{amssymb,amsfonts,amsmath, psfrag}
\usepackage{mathrsfs}
\usepackage{graphics}
\usepackage{graphicx}
\usepackage{epsfig}
\usepackage{mathrsfs}
\makeatletter
\@addtoreset{equation}{section}

\makeatother

\newcommand{\re}{\mathbb{R}}

\newcommand{\N}{\mathbb{N}}

\newcommand{\diag}{\mbox{diag}}
\newcommand{\half}{\frac{1}{2}}

\newcommand{\dt}{\delta}

\def\af{\alpha}
\def\bt{\beta}

\newcommand{\reff}[1]{(\ref{#1})}

\newcommand{\bea}{\begin{eqnarray}}
\newcommand{\eea}{\end{eqnarray}}
\newcommand{\be}{\begin{equation}}
\newcommand{\ee}{\end{equation}}

\newcommand{\baray}{\begin{array}}
\newcommand{\earay}{\end{array}}

\newcommand{\bsry}{\begin{subarray}}
\newcommand{\esry}{\end{subarray}}

\newcommand{\bca}{\begin{cases}}
\newcommand{\eca}{\end{cases}}

\newcommand{\bcen}{\begin{center}}
\newcommand{\ecen}{\end{center}}

\newcommand{\bnum}{\begin{enumerate}}
\newcommand{\enum}{\end{enumerate}}

\newcommand{\bit}{\begin{itemize}}
\newcommand{\eit}{\end{itemize}}

\newcommand{\bbm}{\begin{bmatrix}}
\newcommand{\ebm}{\end{bmatrix}}

\newcommand{\bmx}{\begin{matrix}}
\newcommand{\emx}{\end{matrix}}

\newcommand{\bpm}{\begin{pmatrix}}
\newcommand{\epm}{\end{pmatrix}}

\newtheorem{theorem}{Theorem}[section]
\newtheorem{definition}[theorem]{Definition}

\newtheorem{remark}[theorem]{Remark}
\newtheorem {fact}[theorem]{Fact}
\newtheorem {exm}[theorem]{Example}
\newtheorem {prop}[theorem]{Proposition}
\newtheorem {alg}[theorem]{Algorithm}
\newtheorem {ass}[theorem]{Assumption}

\numberwithin{equation}{section}

\begin{document}

\title{The Split Feasibility Problem with Polynomials}

\author{Jiawang Nie}
\address{Department of Mathematics,
University of California San Diego,
9500 Gilman Drive, La Jolla, CA, USA, 92093.}
\email{njw@math.ucsd.edu}

\author{Jinling Zhao}
\address{
School of Mathematics and Physics, University of Science and
Technology Beijing, Beijing 100083, China.}
\email{jlzhao@ustb.edu.cn}


\date{}


\begin{abstract}
This paper discusses the split feasibility problem with polynomials.
The sets are semi-algebraic, defined by polynomial inequalities.
They can be either convex or nonconvex, either feasible or infeasible.
We give semidefinite relaxations for
representing the intersection of the sets.
Properties of the semidefinite relaxations are studied.
Based on that, a semidefinite relaxation algorithm
is given for solving the split feasibility problem.
Under a general condition, we prove that:
if the split feasibility problem is feasible,
we can get a feasible point; if it is infeasible,
we can obtain a certificate for the infeasibility.
Some numerical examples are given.
\end{abstract}

\maketitle

\section{Introduction}
\hspace*{0.15in}

The split feasibility problem (SFP) can be stated as follows:
for two given sets $C \subset \re^n$, $Q \subset \re^m$,
and a given matrix $A \in \re^{m \times n}$,
find a point $x^*$ such that
\be  \label{x:in:CQ}
x^*\in C, \quad Ax^*\in Q.
\ee
Here, ${\mathbb R}^n$ denotes the $n$-dimensional Euclidean space
over the real field.
The SFP was originally introduced by Censor and Elfving \cite{Cen1}
for modeling phase retrieval problems.
It can serve as a unified tool for modeling
many different inverse problems, such as image reconstruction,
signal processing and intensity-modulated radiation therapy problems.
We refer to \cite{Byr1,Byr2,Cen1,Cen3} and the references therein
for related work on split feasibility problems.
Later, the SFP was generalized to the multiple-sets split feasibility
problem (MSFP), which was introduced by Censor et al. \cite{Cen2}.
The MSFP can be stated similarly as follows:
for given sets
$C_1, \ldots, C_r \subset \re^n$,
$Q_1, \ldots, Q_t \subset \re^m$
and $A \in \re^{m \times n}$,
find a point $x^*$ such that
\be
x^*\in  \bigcap\limits_{i=1}^{r}C_i \ \ {\rm such\ that}\ \
Ax^*\in \bigcap\limits_{j=1}^{t}Q_j .
\ee
In particular, if $t=r=1$, the MSFP collapses to the SFP.

In the prior existing literature, $C, Q, C_i, Q_j$
are often assumed to be nonempty closed convex sets.
The projection type methods have been widely used
for solving the SFPs and MSFPs. We refer to the work
Byrne~\cite{Byr1,Byr2}, Censor et. al.~\cite{Cen1,Cen2,Cen3,Cen4}
and others \cite{Dang1,Eche,Qu1,Xu,Yang,Zhao}.

The CQ algorithm, proposed by Byrne~\cite{Byr1,Byr2},
is a classic method for solving the SFP.
Many other methods for solving SFPs and MSFPs
can be viewed as variations of it.
The CQ method has the basic iteration form:
\[
x^{k+1}=P_C \big(x^k-\gamma A^T (I-P_Q)Ax^k \big),
\]
where $\gamma\in (0,2/\rho(A^TA))$ is a parameter and
$\rho(A^TA)$ denotes the largest eigenvalue of $A^TA$.
The $P_C$ (resp., $P_Q$) stands for the projection
onto the set $C$ (resp., $Q$).
Usually, the performance of CQ type methods
depends on the initial point $x^0$, the choice of $\gamma$
and the geometry of the sets. Moreover,
they also require the sets $C$, $Q$, $C_i$, $Q_j$ to be convex.
Although a lot of progresses have been made,
there still exist computational challenges
for CQ type methods.

In this paper, we focus on the split feasibility problem with polynomials,
i.e., $C,Q$ are semi-algebraic sets given as
\be \label{df:setC}
C:=\{x\in {\mathbb{R}}^n|\ f_i(x)\ge 0,\ i=1,...,r\},
\ee
\be \label{df:setQ}
Q:=\{y\in {\mathbb{R}}^m|\ g_j(y)\ge 0, j=1,...,t\}.
\ee
In the above, the functions $f_i(x),\ g_j(y)$ are real multivariate
polynomials. The problem can also be formulated as a MSFP, with
\[
C_i :=\{x\in {\mathbb{R}}^n|\ f_i(x)\ge 0\},
\quad
Q_j :=\{y\in {\mathbb{R}}^m|\ g_j(y)\ge 0\}.
\]
In this paper, the sets $C$, $Q$, $C_i$ and $Q_j$
are not necessarily assumed to be convex.

%
%

Since the sets are defined by polynomial inequalities,
we propose semidefinite programming (SDP) \cite{SDP} relaxation methods
for solving the SFP, using Lasserre type moment relaxations.
The properties of the SDP relaxations are studied.
Under a general condition, we show that:
if the SFP is feasbile
(i.e., it has at least one solution),
then we can compute a point $x^*$ satisfying \reff{x:in:CQ};
if the SFP is infeasible
(i.e., there is no point $x^*$ satisfying \reff{x:in:CQ}),
then we can obtain a certificate for the infeasibility.

The paper is organized as follows. In Section 2,
we give semidefinite relaxations for the intersection $C \cap A(Q)$.
In Section 3, we give an algorithm for solving the SFP
and prove its convergence properties.
In Section 4, we report some numerical experiments.

\section{Semidefinite relaxations}

\hspace*{0.15in}
In this section, we give semidefinite relaxations
for the sets in the SFP. To do this, we need some tools
from polynomial optimization \cite{Las01,Las09,Las15,Lau09,Lau14}.

\subsection*{Notation}

The symbol $\mathbb N$ stands
for the set of nonnegative integers, and $\mathbb R$
for the set of real numbers. For $s\in {\mathbb R}$,
$\lceil s\rceil$ denotes the smallest integer not smaller than $s$.
For $x :=(x_1, \ldots, x_n) \in {\mathbb R}^n$
and $\af := (\af_1, \ldots, \af_n) \in \N^n$, denote
\[
x^\alpha := x_1^{\alpha_1}\cdots x_n^{\alpha_n}, \quad
|\alpha|:=\alpha_1+\cdots+\alpha_n.
\]
The $x_i$ (resp., $\af_i$)
denotes the $i$-th entry of $x$ (resp., $\af$).
For a degree $d>0$, denote
\[
{\mathbb{N}}_d^n \, := \,
\{\alpha\in {\mathbb{N}}^n|\ |\alpha|\le d\}.
\]
Denote by $[x]_d$ the column vector of
all monomials in $x$ and of degrees at most $d$
(they are ordered in the graded lexicographical ordering), i.e.,
\[
[x]_d  := [1, \, x_1, \ldots, x_n, \, x_1^2,\,  x_1x_2,\,
\ldots,\, x_{n-1}x_n^{d-1}, x_n^d \,]^T.
\]
The symbol ${\mathbb R}[x]$ denotes the ring of polynomials in
$x$ with real coefficients, and ${\mathbb R}[x]_d$
is the space of real polynomials in $x$
with degrees at most $d$. For a polynomial $p$,
$\deg(p)$ stands for its total degree.
For a symmetric matrix $X$, $X\succeq 0$ means
$X$ is positive semidefinite.
For a vector $x$, $\| x \|$ denotes its Euclidean norm.
In the space $\re^n$, $e$ denotes the vector of all ones,
while $e_i$ denotes the $i$th unit vector in the canonical basis.
The $0$ denotes the zero vector.
Denote by $I$ the identity matrix,
when the dimension is clear in the context.

\subsection{Localizing matrices}

The set $\re^{ \N^n_d }$ is the space of all real vectors
that are labeled by $\af \in \N_d^n$. That is, every
$y \in \re^{ \N^n_d }$ can be labeled as
\[
y \, = \, (y_\af)_{ \af \in \N_d^n }.
\]
In some literature, e.g., \cite{Nie3}, such $y$ is called
{\it truncated multi-sequences} (tms) of degree $d$.
For a polynomial $f \in \re[x]_{2k}$, the product
$f(x)[x]_s[x]_s^T$
is a symmetric matrix polynomial of length $\binom{n+s}{s}$,
where $s = \lceil k- \deg(f)/2 \rceil$.
Expand it as
\[
f(x)[x]_s[x]_s^T \, = \, \sum_{ \af \in \N_{2k}^n }
x_\af  F_\af,
\]
for symmetric matrices $F_\af$.
For $y \in \re^{ \N^n_d{2k} }$, define the matrix
\be \label{df:Lf[y]}
L_{f}^{(k)}[y] \, := \, \sum_{ \af \in \N_{2k}^n }
y_\af  F_\af.
\ee
It is called the $k$th {\it localizing matrix} generated by $f$ and $y$.
Clearly, for any fixed $f$, $L_{f}^{(k)}[y]$
is linear in $y$; for any fixed $y$, $L_{f}^{(k)}[y]$
is linear in $f$ (with fixed degree).

\begin{fact}
If $f(u) \geq 0$ and $y = [u]_{2k}$, then
\[ L_{f}^{(k)}[y] = f(u) [u]_s[u]_s^T \succeq 0. \]
\end{fact}

\begin{exm}
For the case $n=2$, $k=2$
and $f= 1 - x_1^2-x_2^2$, we have
\[
L_f^{(2)}[y]=\left [\begin{matrix}
y_{00}-y_{20}-y_{02} &  y_{10}-y_{30}-y_{12} &  y_{01}-y_{21}-y_{03} \\
y_{10}-y_{30}-y_{12} &  y_{20}-y_{40}-y_{22} &  y_{11}-y_{31}-y_{13} \\
y_{01}-y_{21}-y_{03} &  y_{11}-y_{31}-y_{13} &  y_{02}-y_{22}-y_{04} \\
\end{matrix}\right ].
\]
\end{exm}

\subsection{Semidefinite relaxations}

The split feasibility problem is to find a point $x^*$ such that
\be
x^*\in C:=\bigcap_{i=1}^t C_i, \quad
Ax^*\in Q := \bigcap_{j=1}^t Q_j,
\ee
where $A \in \re^{m \times n}$ is given and
\[
C_i =\{x\in {\mathbb{R}}^n|\ f_i(x)\ge 0 \}, \quad
Q_j =\{y\in {\mathbb{R}}^m|\ g_j(y)\ge 0 \},
\]
for polynomials $f_i(x),\ g_j(y)$. For convenience, denote
\be \label{df:hj(x)}
h_j(x) \, := \, g_j(Ax).
\ee
Each $h_j$ is a polynomial in $x$. Let
\be \label{df:H}
H:=\{x\in {\mathbb{R}}^n|\ h_j(x)\ge 0, j=1,...t\}.
\ee
The SFP is equivalent to finding a point
\[
x^* \, \in \, C \cap H.
\]
Denote the degrees
\be \label{def:d}
\baray{rcll}
d_{f,i} &:=& \lceil \deg(f_i)/2\rceil,  &
d_{h,j} := \lceil \deg(h_j)/2\rceil, \\
d &:=& \max\limits_{i,j}\{d_{f,i},d_{h,j} \}. &
\earay
\ee

\begin{fact}
For all $k\ge d$ and for all $u\in C \cap H$,
we have
\[ f_i(u), h_j(u)\ge 0 \]
for all $i,j$. Hence,
\[
\baray{c}
f_i(x) [x]_{k-d_{f,i} } \big( [x]_{k-d_{f,i} } \big)^T  \succeq 0, \quad
h_j(x) [x]_{k-d_{h,j} } \big( [x]_{k-d_{h,j} } \big)^T  \succeq  0.
\earay
\]
This implies that if $y = [u]_{2k}$ and $u \in C \cap H$, then
\[
L_{f_i}^{(k)}[y] \succeq 0, \quad
L_{g_j}^{(k)}[y] \succeq 0.
\]
Let $f_0 = 1$, then
$
f_0(x)[x]_{k}[x]_{k}^T\succeq 0
$
for all $x \in \re^n$. So,
\[
L_{1}^{(k)}[y] \succeq 0
\]
for all $y = [u]_{2k}$. It is called the moment matrix of $y$.
\end{fact}

Note that
\[
C \cap H =
\left\{ x \in \re^n
\left|\baray{c}
f_1(x) \geq 0, \ldots, f_r(x) \geq 0 \\
h_1(x) \geq 0, \ldots, h_t(x) \geq 0
\earay \right.
\right\}.
\]
So, $C \cap H$ is always contained in the set (note $f_0=1$)
\be \label{SDr:Sk}
S_k := \left\{ x \in \re^n
\left| \baray{c}
\exists y \in \re^{\N_{2k}^n}, \, y_0 = 1, \\
x = (y_{e_1}, \ldots, y_{e_n} ),  \\
L_{f_i}^{(k)}[y] \succeq 0 \,( i=0, \ldots, r), \\
L_{h_j}^{(k)}[y] \succeq 0 \,( j=1, \ldots, t) \\
\earay \right.
\right\},
\ee
for all $k \geq d$. Each $S_k$ is the projection of
a set in $\re^{\N_{2k}^n}$
that is defined by linear matrix inequalities.
It is a {\it semidefinite relaxation} of $C \cap H$,
because $C \cap H \subseteq S_k$ for all $k \geq d$.
It holds the nesting containment relation \cite{Nie1}
\be \label{nest:cont}
S_d \supseteq S_{d+1} \supseteq \cdots \supseteq C\cap H.
\ee

\subsection{Some basic properties}

The first one is about the feasibilities and infeasibilities between
$C \cap H$ and the semidefinite relaxations $S_k$ in \reff{SDr:Sk}.

\begin{prop}  \label{prop:feas}
Let $C,H$ be the sets in the above.
If the intersection $C \cap H \ne \emptyset$,
then the semidefinite relaxation
$S_k \ne \emptyset$ for all $k \geq d$.
Therefore, if $S_k = \emptyset$ for some $k \geq d$,
then $C \cap H = \emptyset$,
i.e., the split feasibility problem is infeasible.
\end{prop}
\begin{proof}
For all $u \in C \cap H $ and $k \geq d$,
the tms $y:= [u]_{2k}$ satisfies the linear matrix inequalities
(for all $i,j$)
\[
L_{f_i}^{(k)}[y] \succeq 0 \, \quad
L_{h_j}^{(k)}[y] \succeq 0.
\]
Note that $y_0=1$ and $u = (y_{e_1}, \ldots, y_{e_n} )$,
so $u \in S_k$. Therefore, $C \cap H \ne \emptyset$ implies
$S_k \ne \emptyset$ for all $k \geq d$.
Consequently, if there exists $k \ge d$ such that $S_k = \emptyset$,
then we must have $C \cap H = \emptyset$.
\end{proof}

\begin{remark}
In Proposition~\ref{prop:feas},
we do not require the set $C$ or $H$ to be convex.
Moreover, neither $f_i$ nor $h_j$
is assumed to be concave.
\end{remark}

Next, we give conditions for the equality $S_k = C \cap H$.
A typical one is the sos-convexity/concavity.

\begin{definition}(\cite{Nie1})
A polynomial $f(x) \in \re[x]$ is called \textit{sos-convex} if
there exists a matrix polynomial $P(x) \in \re[x]^{\ell \times n}$
such that ($\ell$ may be different from $n$)
\[
\nabla^2 f(x) :=
\Big(\frac{\partial^2 f}{\partial  x_i \partial x_j}
\Big)_{i,j=1}^n
= P(x)^TP(x).
\]
Similarly, $f(x)$ is called
\textit{sos-concave} if $-f(x)$ is sos-convex.
\end{definition}

A polynomial $p\in \re[x]$ is said to be sos
if $p=p_1^2 +\cdots +p_k^2$,
for some real polynomials $p_1,\ldots,p_k \in \re[x]$.
We refer to \cite{BPR,Rez} for sos polynomials.

\begin{theorem}  \label{CH=Sk:sos}
Let $C, H, S_k$ be the sets in the above.
Assume that the polynomials $f_i(x)$ and $h_j(x)$ are
all sos-concave. Then, $S_k = C \cap H$
for all $k \geq d$.
\end{theorem}
\begin{proof}
In \reff{nest:cont}, we have already seen that
$C \cap H \subseteq S_k$ for all $k \geq d$.
We need to prove the reverse containment
$
C \cap H \supseteq S_k.
$
Choose an arbitrary point $u \in S_k$,
with $u =(y_{e_1}, \ldots, y_{e_n})$ and
$y$ satisfying the conditions in \reff{SDr:Sk}.
We show $u \in C \cap H$ as follows.
Consider the new polynomial:
\[
\tilde{f}_i(x) \, := \, -f_i(x)+f_i(u)+\nabla f_i(u)^T(x-u).
\]
It is a polynomial in $x$, for fixed $u$. Note that
\[
\tilde{f}_i(u) = 0, \quad  \nabla \tilde{f}_i(u) = 0.
\]
By Lemma~8 of \cite{Nie1}, we know that
$\tilde{f}_i$ is an sos polynomial, say,
\[
\tilde{f}_i = p_1(x)^2  + \cdots + p_N(x)^2,
\]
for some real polynomials $p_1,\ldots, p_N \in \re[x]_k$.
Define the linear functional
\be \label{lf:Ry}
\mathscr{R}_y: \, \re[x]_{2k} \to \re, \quad
\mathscr{R}_y ( x^\af ) = y_\af,
\ee
for all $\af \in \N_{2k}^n$.
Since $L_{1}^{(k)}[y] \succeq 0$ (note $f_0 = 1$), we can see that
\[
\mathscr{R}_y (\tilde{f}_i )
= \sum_{\ell =1}^N  vec(p_\ell)^T \Big( L_{1}^{(k)}[y] \Big)  vec(p_\ell)
\geq 0.
\]
Here, $vec(p_\ell)$ denote the coefficient vector of $p_\ell$. So,
\[
\mathscr{R}_y (\tilde{f}_i ) = - \mathscr{R}_y(f_i) + f_i(u) \geq 0,
\]
because $\mathscr{R}_y \big(\nabla f_i(u)^T(x-u) \big)=0$.
Also note that $\mathscr{R}_y(f_i)$ is the $(1,1)$-entry of the matrix
$L_{f_i}^{(k)}[y]$. Thus, $L_{f_i}^{(k)}[y]\succeq 0$
and the above imply that
\[
f_i(u) \geq \mathscr{R}_y(f_i) \geq 0.
\]
This is true for all $i=1,\ldots, r$.
In the same way, we can prove that $h_j(u) \geq 0$ for all $j$.
Therefore, $u \in C \cap H$ for all $u \in S_k$,
i.e., $S_k \subseteq C \cap H$.
\end{proof}

\begin{remark}
In Theorem~\ref{CH=Sk:sos},
we do not assume that the sets $C,H$
have nonempty interiors. In other words, even if
either $C$ or $H$ has empty interior,
the conclusion $C \cap H = S_k$ is still true,
under the sos-concavity assumption on $f_i,h_j$.
This is stronger than the results in \cite{Nie1}.
\end{remark}

When the polynomials $f_i, h_j$ are not sos-concave,
there still exist general conditions ensuring $C\cap H = S_k$.
We refer to \cite{HN2,Nie1,Las09MP,Las09siopt} for related work.

\section{An algorithm for solving the SFP}
\setcounter{equation}{0}

Let $C,Q$ be the sets as in
\reff{df:setC}-\reff{df:setQ}, defined
by polynomials $f_i, g_j$, with a given matrix $A \in \re^{m\times n}$
and $h_j(x) = g_j(Ax)$.
The set $H$ is as in \reff{df:H}. In this section,
the sets $C,Q, H$ are not assumed to be convex,
and none of $f_i, h_j$ is assumed to be concave.
The split feasibility problem is to find a point
$x^* \in C$ such that $Ax^* \in Q$, which
is equivalent to $x^* \in C \cap H$.

In this section, we propose an algorithm for solving the SFP,
based on the semidefinite relaxations in \reff{SDr:Sk}.
If it is feasible, we want to get a point $x^* \in C \cap H$;
if it is infeasible, we want to obtain a certificate for the infeasibility.

Let $d$ be the degree as in \reff{def:d}.
Choose a generic vector $\xi \in \re^{ \N_{2d}^n }$
such that $\| \xi\| \le \half$.
Consider the polynomial optimization problem
\be \label{POP:xi}
\left\{\baray{rl}
\min &  c(x):= \| [x]_d \|^2 + \xi^T [x]_{2d}, \\
\mbox{s.t.} & f_1(x) \geq 0, \ldots, f_r(x) \geq 0, \\
& h_1(x) \geq 0, \ldots, h_t(x) \geq 0.
\earay \right.
\ee
Because $\| \xi\| \le \half$, the objective $c(x)$
is a coercive function, i.e., the sublevel set
$\{ x \in \re^n: \, c(x) \leq M\}$
is compact for all $M$, so \reff{POP:xi}
must have a global minimizer whenever the SFP is feasible.
For solving \reff{POP:xi},
the Lasserre type moment relaxation of order $k$ ($\geq d$) is
\be  \label{Las:POP}
\left\{\baray{rl}
\min &  \sum_{|\bt| \leq d} \, y_{2\bt} +
          \sum_{|\af| \leq 2d} \, \xi_{\af} y_{\af} \\
\mbox{s.t.} & L_{f_i}^{(k)}[y] \succeq 0 \, ( 1 \le i \le r), \\
& L_{h_j}^{(k)}[y] \succeq 0 \, ( 1 \le j \le t), \\
& L_{1}^{(k)}[y] \succeq 0, \,  y_0 = 1, \, y \in \re^{\N_{2k}^n}. \\
\earay \right.
\ee
The set of $y$ satisfying \reff{Las:POP}
is the same as the set of $y$ satisfying \reff{SDr:Sk}.
So, \reff{Las:POP} is a semidefinite relaxation
for solving the SFP.

\begin{alg}  \label{alg:SFP}
Let $f_i, h_j$ be as above. Set $k := d$.
\bit

\item [Step~1] Solve the semidefinite program \reff{Las:POP}.
If it is infeasible, then the SFP is infeasible and stop;
otherwise, solve it for an optimizer $y^*$ if it exists
and then go to Step~2.

\item [Step~2] Let $u^k := (y^*_{e_1}, \ldots, y^*_{e_n})$.
If $f_i(u^k) \geq 0$ and $h_j(u^k) \geq 0$ for all $i,j$,
then $u^k$ is a solution of the SFP;
otherwise, let $k:=k+1$ and go to Step~1.

\eit

\end{alg}

The conclusion about infeasibility in Step~1 is justified
by Proposition~\ref{prop:feas}, while
the conclusion of Step~2 is very straightforward.
When the SFP is feasible, the semidefinite program
\reff{Las:POP} does not necessarily have an optimizer.
To guarantee the solvability of \reff{Las:POP},
we need the following assumption.

\begin{ass} \label{ass:AC}
There exist sos polynomials $a_0, a_1, \ldots, a_r$,
$b_1, \ldots, b_t$ and a real number $R>0$ such that
\be \label{ac:aifi:bjhj}
R- \| [x]_d\|^2 = a_0 +a_1 f_1+\cdots +a_rf_r + b_1h_1 +\cdots +b_th_t
\ee
and $\deg(a_i f_i), \deg(b_jh_j) \leq 2d$ for all $i,j$.
\end{ass}

Assumption~\ref{ass:AC} implies that
$\| [x]_d \|^2 \leq R$ for all $x \in C \cap H$,
so the intersection $C \cap H$ is bounded.
The reverse is not necessarily true.
However, when $C \cap H$ is bounded, say, $\| [x]_d \|^2 \leq R$,
we can add the polynomial $R - \| [x]_d \|^2$ to
the set $\{f_1, \ldots, f_r\}$,
while the intersection $C \cap H$ is not changed.
If the degree bounds on $a_i,b_j$ are removed,
Assumption~\ref{ass:AC}
becomes the classical {\it archimedean condition}
\cite{Las01,Putinar}.
The convergence of Algorithm~\ref{alg:SFP}
is summarized in the following theorem.

\begin{theorem}  \label{thm:sfpcvg}
Let $C,Q,f_i,h_j$ be as above.
If Assumption~\ref{ass:AC} holds,
then we have the properties:
\bit

\item [(i)]
If the SFP is infeasible, then
the semidefinite program \reff{Las:POP}
must be infeasible for all $k$ big enough.

\item [(ii)] If the SFP is feasible and
$\xi$ is generically chosen such that $\| \xi \| \le \half$,
then \reff{Las:POP} has an optimizer for all $k\geq d$
and the sequence $\{u^k\}_{k=d}^{\infty}$
converges to a solution $x^*$ of the SFP.

\eit
\end{theorem}
\begin{proof}
(i) For convenience of notation, let
$f_{r+j} = h_j$ for $j=1,\ldots, t$. Then
\[
C \cap H = \{ x \in \re^n: \, f_i(x) \geq 0, i=1,\ldots, r+t\}.
\]
If the SFP is infeasible, the intersection $C \cap H $ is empty.
By the Positivstellensatz \cite{BCR}, there exist sos polynomials
$s_\dt$ ($\dt \in \{0,1\}^{r+t}$) such that
\[
-2 = \sum_{ \dt \in \{0,1\}^{r+t} }  s_\dt
f_1^{\dt_1} \cdots f_{r+t}^{\dt_{r+t} }.
\]
Denote the polynomial
\[
p \, := \, 1 + \sum_{ \dt \in \{0,1\}^{r+t} }  s_\dt
f_1^{\dt_1} \cdots f_{r+t}^{\dt_{r+t} } ,
\]
which is the constant polynomial $-1$.
Clearly, $p$ is positive on $C\cap H$. Denote
\[
F \, := \, (f_1, \ldots, f_{r+t}).
\]
Define the $k$th quadratic module of $F$ ($f_0=1$)
\[
\mbox{Qmod}_k(F) :=
\left\{
\left. \sum_{i=0}^{r+t} s_i f_i \right| s_i \, \mbox{ is sos},
\deg(s_if_i) \leq 2k
\right\}.
\]
It is a convex cone. By Putinar's Positivstellensatz~\cite{Putinar},
we have $p \in \mbox{Qmod}_k(F)$
(i.e., $-1 \in \mbox{Qmod}_k(F)$ because $p=-1$),
when $k$ is sufficiently large.
The dual optimization problem of
\reff{Las:POP} can be shown to be (cf. \cite{Las01})
\be  \label{max:gm}
\max \quad \gamma  \quad
\mbox{s.t.} \quad c - \gamma \in \mbox{Qmod}_k(F).
\ee
Every feasible $\gamma$ in \reff{max:gm}
is a lower bound for the objective value
of \reff{Las:POP}, whenever $y$ is feasible.
This is the so-called weak duality.
For all $k$ sufficiently large such that $-1 \in \mbox{Qmod}_k(F)$,
the problem \reff{max:gm} is unbounded from above,
by which the weak duality implies that
\reff{Las:POP} is infeasible.

(ii) Let $\mathscr{R}_y$ be the linear functional defined as in \reff{lf:Ry}.
By the equality \reff{ac:aifi:bjhj} in Assumption~\ref{ass:AC},
we can get (note $f_0=1$, $y_0=1$)
\[
R - \sum_{|\bt| \leq d} y_{2\bt} = \mathscr{R}_y( R - \| [x]_d \|^2)
= \sum_{i=0}^r \mathscr{R}_y( a_i f_i ) + \sum_{j=1}^t \mathscr{R}_y( b_j h_j ).
\]
Since $a_i = p_1^2+\cdots + p_\ell^2$ is sos, one can verify that
\[
\mathscr{R}_y( a_i f_i ) = \sum_{j=1}^\ell \mathscr{R}_y( p_j^2 f_i ) = \sum_{j=1}^\ell
vec(p_j)^T \Big(  L_{f_i}^{(k)}[y] \Big)  vec(p_j) \geq 0,
\]
because $ L_{f_i}^{(k)}[y] \succeq 0$.
(The $vec(p_j)$ denotes the coefficient vector of $p_j$.)
Similarly, we can show that
$\mathscr{R}_y( b_j h_j ) \geq 0$,
because $ L_{h_j}^{(k)}[y] \succeq 0$. So,
$
R \geq \sum_{|\bt| \leq d} y_{2\bt}.
$
Moreover, for all $x^\theta$ with $|\theta| \leq k-d$,
\[
R x^{2\theta} - \| x^\theta [x]_d \|^2 =
\sum_{i=0}^r (a_i x^{2\theta}) f_i +
\sum_{j=1}^t \mathscr{R}_y (b_j x^{2\theta})  h_j.
\]
By the same argument, we can show that
\be  \label{R-x2:sos}
R y_{2\theta} \geq \sum_{|\bt| \leq d} y_{2\bt+2\theta}
\ee
for all $y$ that is feasible in \reff{Las:POP}.
The diagonal entries of $L_{1}^{(k)}[y]$
are precisely $y_{2\bt}$ with $|\bt| \leq k$.
Since $L_{1}^{(k)}[y] \succeq 0$, $y_{2\bt} \geq 0$
for all $|\bt|\leq k$ and $\| y \|^2$
is bounded by the trace of $L_{1}^{(k)}[y]$.
Applying \reff{R-x2:sos} recursively,
one can see that the norm of $y$ can be bounded
by a constant that is only depending on $R,k$.
So, the feasible set of \reff{Las:POP}
is compact, and hence \reff{Las:POP}
must have an optimizer.

Since $\| \xi \| \le \half$, the objective $c(x)$
is a coercive polynomial. When $\xi$ is generically chosen,
the objective $c(x)$ has a unique optimizer over the set
$C \cap H$. By Corollary~3.5 of \cite{Swg05}
or Theorem~3.3 of \cite{FlatTrun}, the sequence $\{u^k\}_{k=d}^{\infty}$
must converge to the unique optimizer of \reff{POP:xi},
which is clearly a solution to the SFP.
\end{proof}

\begin{remark}
Under some general conditions,
Algorithm~\ref{alg:SFP} must terminate in finitely many steps.
To be more precisely, under the standard constraint qualification,
the second order sufficiency and
the strict complementarity conditions
for the optimization problem \reff{POP:xi},
the hierarch of Lasserre type moment relaxations \reff{Las:POP}
must have finite convergence. We refer to \cite{opcd}
for more details about the finite convergence
and refer to \cite{FlatTrun} about how to detect convergence.

In particular, Algorithm~\ref{alg:SFP} terminates at the first iteration
with $k=d$, when the polynomials $f_i, h_j$ are sos-concave.
This is because $C\cap H = S_d$, by Theorem~\ref{CH=Sk:sos},
and hence the set of $u = (y_{e_1},\ldots, y_{e_n})$,
with $y$ satisfying all the constraints in \reff{Las:POP},
is the same as the feasible set of \reff{POP:xi}.
\end{remark}

\section{Numerical experiments}
\setcounter{equation}{0}

This section reports numerical experiments
for solving split feasibility problems with polynomials.
Algorithm~\ref{alg:SFP} is applied to solve them.
It can be implemented conveniently by the software
{\tt GloptiPoly} \cite{Glop}, which calls
{\tt SeDuMi} \cite{SeD1,SeD2} for solving the semidefinite programs.
The computation is implemented in MATLAB R2014
on a laptop with Intel(R) Core(TM) i5-3337U CPU (1.80 GHz).
The computational results are displayed in four decimal digits.
For convenience of expression, we give the sets $C,Q$
by inequalities of the form $-f_i(x) \leq 0$ or $-g_j(y) \leq 0$.
If there is an equality $f(x)=0$,
it can be equivalently expressed as:
$-f(x) \leq 0, f(x) \leq 0$.

The classical CQ algorithm for solving SFPs has the iterative formula
\[
x^{k+1}=P_C(x^k-\gamma A^T (I-P_Q)Ax^k),
\]
for $k=0,\ 1,\ 2,\ \cdots$, where $0<\gamma<2/\rho(A^TA)$
is a parameter and $P_C$, $P_Q$ denote projections onto the sets
$C$, $Q$ respectively. To implement the CQ algorithm,
one needs to compute the projection $P_C(u)$,
which is equivalent to solving the optimization problem
\[
P_C(u)= {  \mbox{arg} \noindent \min\limits_{z\in C}  }
\quad \frac{1}{2}\|z-u\|^2.
\]
The same is true for $P_Q$. This might require a large amount of computations.
In practice, people often apply the relaxed CQ algorithm~\cite{Yang},
in which only the projections onto
hyperplane or half spaces are required.
A typical iterative formula for relaxed CQ algorithm is:
\[
x^{k+1} \, = \, P_{C_k} \big(x^k-\gamma A^T (I-P_{Q_k})Ax^k \big),
\]
where $C_k$ and $Q_k$ are half spaces passing through $x^k$
that contain the sets $C$ and $Q$, respectively.
In Example~\ref{exm:4.1}, we choose
$\gamma=1.8/\rho(A^TA)$. The stopping criterion is:
$f_i(x^k)<10^{-5}$ and $g_j(Ax^k)<10^{-5}$, or $k\ge kmax:=10^6$.
The performance of the relaxed CQ algorithm
usually depends on the initial point $x^0$
and the geometry of the SFP. Moreover, the CQ type algorithms
are not applicable when the sets $C,Q$ are not convex.
In contrast, our Algorithm~\ref{alg:SFP} does not assume
the sets are convex; it can also detect infeasibility.

First, we show an example of comparing
the relaxed CQ algorithm and Algorithm~\ref{alg:SFP}.

\begin{exm}  \label{exm:4.1}
Consider the sets $C,Q$:
\[
C=\left\{x\in {\mathbb R}^3 : \,
\frac{1}{2}(x-e)^TB_1(x-e)+b_1^T(x-e)\le 0 \right\},
\]
\[
Q=\left\{y\in {\mathbb R}^2 :\,
\frac{1}{2}(y-Ae)^TB_2(y-Ae)+b_2^T(y-Ae)\le 0 \right\},
\]
where ($a$ is a parameter)
\[
A=\left [\begin{matrix} 1& 2&3\\ 2& 3
&4\end{matrix}\right ],
B_1=\left [\begin{matrix} 1& 2&2\\ 2& 6 &6\\
2& 6& 7\end{matrix}\right ],
B_2=\left [\begin{matrix} a& 1\\
1&2\end{matrix}\right ],
b_1= \bbm 18 \\ 28 \\ 38\ebm,
b_2 = \bbm 2 \\ 8 \ebm.
\]
We want to find a pint $x^* \in C$ such that $Ax^*\in Q$.
Clearly, $e=(1,1,1)^T$ is such a point.
The matrix $B_1$ is positive semidefinite,
so $C$ is convex. When $a\ge \frac{1}{2}$,
$B_2$ is positive semidefinite and $Q$ is also convex.
For different values of $a$, we compare the performance of
Algorithm~\ref{alg:SFP} and the relaxed CQ algorithm.
We choose the initial point as $(-50, 50, 50)^T$,
which is not close to $(1,1,1)^T$.
The numerical results are reported in Table~\ref{tab:01}.
The time is in seconds.
\begin{table}[htb]
\caption{
A comparison between Algorithm~\ref{alg:SFP} and the relaxed CQ algorithm
}
\begin{tabular}{|c|c|c|c|c|c|c|} \hline
Value of $a$ &\ $a=5$\  &\  $a=50$\  &\ $a=500$  \ & $a=5000$ \ & $a=20000$\\
\hline
Relaxed CQ Alg. &0.1055 &0.4126 &0.6487 &1.4480 &2.7401 \\  \hline
Alg.~\ref{alg:SFP}&0.5931 &0.6237 &0.6538 &0.6756 &0.6697 \\ \hline
\end{tabular}
 \label{tab:01}
\end{table}
For $a \leq 500$, the relaxed CQ algorithm can get a solution faster.
However, for larger values of $a$, e.g., $a =5000$ and $a=20000$,
Algorithm~\ref{alg:SFP} is faster. The time
consumed by Algorithm~\ref{alg:SFP}
does not change much as $a$ increases,
while the time by the relaxed CQ algorithm increases fast.
The semidefinite relaxation method behaves more stably.
\end{exm}

\begin{exm}  \label{exm:4.2}
Consider the sets $C,Q$:
\[
C=\left\{x\in {\mathbb R}^5
\left| \baray{l}
x_1^2+x_2^2-x_3^5+x_4x_5-3 \le 0, \, x_1(x_1-1)=0, \\
x_1^4+x_2^4+x_5^4-2\le 0,\, 3x_2+2\le 0
\earay \right.
\right\},
\]
\[
Q=\{y\in {\mathbb R}^4|\ \frac{1}{2}y^TB_1y+b_1^Ty-1\le 0,
\ \frac{1}{2}y^TB_2y+b_2^Ty-2\le 0\},
\]
where
\[
B_1=\left [\begin{matrix} 1& 4&6.5 &6\\ 4& 2 &0.5 &2.5\\
6& 0.5& 10 & 2.5\\ 6& 2.5& 2.5 &9\end{matrix}\right ], \quad
B_2=\left [\begin{matrix} 18 &12& 7 &19.5 \\
12&2&2.5&7.5\\ 7&2.5&10&14\\ 19.5&7.5&14&18\end{matrix}\right ],
\]
\[
A=\left [\begin{matrix} 2& 5&8 &3 &6\\ 1& 0
&4&2&5\\ 6& 9& 7 &0 &1\\ 0 &2 &1 &0 &3\end{matrix}\right ], \quad
b_1= \bbm 2 \\ 1 \\ 4 \\ 3 \ebm, \quad
b_2= \bbm -1 \\ 3 \\ 0 \\ 5 \ebm.
\]
The sets $C,Q$ are nonconvex.
By Algorithm~\ref{alg:SFP}, we got a feasible point
\[
(0.0000, -0.6667, 0.8611, -0.2181, -0.3341)^T.
\]
It took about $2.95$ seconds.
The CQ type methods are not applicable
because of the nonconvexity of $C,Q$.
\end{exm}

\begin{exm} \label{exm:4.3}
Consider the sets ($R$ is a parameter)
\[
C = \left\{x\in {\mathbb R}^3
\left|\baray{r}
-f_1(x):= x_1^4+x_2^4+x_3^4+2x_1^2x_2^2+x_1^2x_3^2+x_2^2x_3^2 \\
-4x_1-4x_2-4x_3+1\le 0
\earay \right.
\right\},
\]
\[
Q= \left\{y\in {\mathbb R}^3 \mid
(y_1-2)^2+(y_2-2)^2+(y_3-2)^2-R \le 0
\right\}.
\]
The function $-f_1(x)$ is convex over ${\mathbb R}^3$, because
\[
\baray{rcl}
-\nabla^2 f_1(x) & = & \left [
\begin{matrix} 12x_1^2+4x_2^2+2x_3^2& 8x_1x_2&4x_1x_3 \\ 8x_1x_2& 4x_1^2+12x_2^2+2x_3^2
&4x_2x_3\\ 4x_1x_3& 4x_2x_3& 2x_1^2+2x_2^2+12x_3^2
\end{matrix}\right ] \\
& = & 4 \bbm x_1 \\ x_2 \\ x_3 \ebm \bbm x_1 \\ x_2 \\ x_3 \ebm^T +
4 \bbm x_1 \\ x_2 \\ 0 \ebm \bbm x_1 \\ x_2 \\ 0 \ebm^T +
\diag \bbm 4x_1^2+4x_2^2+2x_3^2 \\ 4x_1^2+4x_2^2+2x_3^2
\\ 2x_1^2+2x_2^2+8x_3^2  \ebm
\earay
\]
is positive semidefinite for all $x\in{\mathbb R}^3$.
(The $\diag(w)$ denotes the diagonal matrix
whose diagonal is $w$.)
The sets $C,Q$ are both convex. Consider the matrix $A=I$.
This SFP is equivalent to finding $x^* \in C \cap Q$.
The smallest $R$ for $C \cap Q \ne \emptyset$
is $R_0 \approx 2.0623$, which is the square of the distance between
$(2,2,2)$ and $C$.
For different values of $R$, we apply Algorithm~\ref{alg:SFP}
to solve the SFP. The results are shown in Table~\ref{tab:02}.
\begin{table}[htb]
\caption{Computational results for Example~\ref{exm:4.3} }
\begin{tabular}{|r|c|c|} \hline
$R$ \, & feasibility  &  a solution $x^*$ \\
\hline
$4.00$&\ feasible\  &\  $(0.8228,0.8604,0.8531)$\  \\
\hline
$3.00$&\ feasible\  &\  $(1.0012,0.9655,1.0345)$\ \\
\hline
$2.07$&\ feasible\ &\  $(1.1813,1.1285,1.1998)$\ \\
\hline
$2.06$&\ infeasible\ &\  none\ \\
\hline
$2.00$&\ infeasible\ &\  none\ \\
\hline
$1.00$&\ infeasible\ &\  none\ \\
\hline
\end{tabular}
\label{tab:02}
\end{table}
\end{exm}

\begin{exm} \label{exm:4.4}
Let $C$ be the same set as in Example~\ref{exm:4.3} and $Q$ be
\[
Q \, := \, \left\{y\in {\mathbb R}^3 \mid
(y_1-2)^2+(y_2-2)^2+a(y_3-2)^2-1.5\le 0 \right\},
\]
where $a$ is a parameter.  The matrix $A$ is also the identity.
When $a=1$, the SFP is infeasible because $1.5<R_0$.
However, the feasibility changes as $a$ varies.
In Table~\ref{tab:03}, we list some values of $a$
such that the SFP is feasible/infeasible.
\begin{table}[htb]
\caption{Computational results for Example~\ref{exm:4.4} }
\begin{tabular}{|r|c|c|} \hline
$a$\,\,\, &\ feasibility &\  a solution $x^*$\  \\ \hline
$1.00$&\ infeasible\  &\  none  \\ \hline
$0.50$&\ infeasible\  &\  none \\ \hline
$0.25$&\ feasible\ &\  $(1.2745,1.2381,0.7461)$\ \\
\hline
$0.10$&\ feasible\ &\  $(1.2237,1.2001,0.3954)$\ \\
\hline
$0.00$&\ feasible\ &\  $(1.1187,1.1495,-0.0248)$\ \\
\hline
$-5.00$&\ feasible\ &\  $(0.0912,0.0374,0.1216)$\ \\
\hline
\end{tabular}
\label{tab:03}
\end{table}
The maximum value for $a$ such that
$C \cap Q \ne \emptyset$ is around $0.2786$,
which can be found by maximizing $a$
subject to the constraints in $C,Q$.
%
%
\end{exm}

\begin{exm} \label{exm:4.5}
Consider the sets $C,Q$:
\[
C=\left\{x\in {\mathbb R}^3
\left|\baray{l}
5x_1^{10}+3x_1^5x_3+x_2^4+x_3^2+8x_3+1 \le 0,\\
x_1^4+5x_1^2x_2^2-8x_1x_2+3x_2x_3^3+x_2^4+x_3^4-1\le 0
\earay \right.
\right\},
\]
\[
Q=\{y\in {\mathbb R}^2|\ \frac{1}{2}y^TB_1y+b_1^Ty\le 0, \,
 \frac{1}{2}y^TB_2y+b_2^Ty-1\le 0\},
\]
with
\[
A= \bbm 1& 2 &3 \\ 0& 1& 2\ebm, \,
B_1=\bbm 1& 1.5\\ 1.5& 5\ebm, \,
B_2=\left [\begin{matrix} 5 & 0.5 \\
0.5 & 2 \end{matrix}\right ], \,
b_1= \bbm 2 \\ 8 \ebm, b_2= \bbm 10 \\ 5 \ebm.
\]
The sets $C,Q$ are nonconvex.
By Algorithm~\ref{alg:SFP}, we know the SFP is feasible,
and we got a solution $(0.0138,-0.0128,-0.1270)^T$.
It took about $2$ seconds.
\end{exm}

\begin{exm} \label{exm:4.6}
Consider the sets ($R$ is a parameter)
\[
C= \left\{x\in {\mathbb R}^2
\left|\baray{r}
f(x) := x_1^5-10x_1^4x_2+8x_1^2x_2^3-6x_1^2x_2^2+5x_1^3\\
-7x_1^2x_2+3x_1x_2^2-9x_2^3+2x_1^2+1 \le 0
\earay\right.
\right\},
\]
\[
Q=\{y\in {\mathbb R}^2|\ y_1^2+y_2^2-R \le 0\}.
\]
The matrix $A=I$.
%
%
%
%
The set $C$ is nonconvex and unbounded.
We apply Algorithm~\ref{alg:SFP} to solve the SFP,
with different values of $R$. The numerical results are
stated in the following Table~\ref{tab:04}.
\begin{table}[htb]
\caption{Computational results for Example~\ref{exm:4.6} }
\begin{tabular}{|r|c|c|} \hline
$R$ \, &\ feasibility   &\  a solution $x^*$  \\ \hline
$100.0$&\ feasible\  &\ $(-0.0845,0.4690)$\ \\ \hline
$10.0$&\ feasible\  &\ $(-0.0126,0.4793)$\ \\ \hline
$1.0$&\ feasible\ &\  $(-0.2343,0.4292)$\ \\ \hline
$0.5$&\ feasible\ &\  $(-0.2128,0.4370)$\ \\ \hline
$0.2$&\ infeasible\ &\  none\ \\ \hline
$0.1$&\ infeasible\ &\  none\ \\ \hline
\end{tabular}
\label{tab:04}
\end{table}
\end{exm}

\begin{exm} \label{exm:4.7}
Let $A=I$ and $C,Q$ be the sets ($a$ is a parameter):
\[
C=\left\{x\in \mathbb{R}^2
\left| \,\frac{1}{9}x_1^2+\frac{1}{4}x_2^2-1\le 0, \right.
\ x_1^2+x_2^2-1\ge 0
\right\},
\]
\[
Q= \left\{y\in \left. \mathbb{R}^2 \, \right|
\, y_2-y_1\le 0,\ y_1\le 2,\ y_2\ge a
\right\}.
\]
The set $C$ is nonconvex.
The intersection $C \cap Q$ is nonconvex for $a< \sqrt{2}/2$.
The SFP is infeasible for $a>\sqrt{36/13}\approx 1.6641$.
%
%
By Algorithm~\ref{alg:SFP},
for different values of $a$, we solve the SFP.
The computational results are in Table~\ref{tab:05}.
\begin{table}[htb]
\caption{Computational results for Example~\ref{exm:4.7} }
\begin{tabular}{|r|c|c|c|} \hline
$a$ \, &  $C\cap Q$ \ &\ feasibility \  &\ a solution $x^*$  \\
\hline
$-2.0$&\ nonconvex\ &\ feasible\  &\  $(0.4520,-0.8920)$\  \\
\hline
$-1.5$&\ nonconvex\ &\ feasible\  &\  $(0.4059,-0.9139)$\ \\
\hline
$-1.0$&\ nonconvex\ &\ feasible\ &\  $(-0.0922,-0.0922)$\ \\
\hline
$0.0$&\ nonconvex\ &\ feasible\ &\  $(1.0000,0.0000)$\ \\
\hline
$\sqrt{2}/2$&\ convex\ &\ feasible\ &\  $(0.7071,0.7071)$\ \\
\hline
$1.8$& convex\ &\ infeasible\ &\  none \\
\hline
\end{tabular}
\label{tab:05}
\end{table}
%
%
\end{exm}

\section{Conclusions}

This paper discusses the split feasibility problem with polynomials.
The sets are semi-algebraic sets, defined by polynomial inequalities.
But they are allowed to be nonconvex or even infeasible.
Semidefinite relaxations are proposed for 
representing the intersection of the sets.
We gave conditions that guarantee these relaxations
are exact for representing the intersection. Based on these relaxations,
Algorithm~\ref{alg:SFP} is proposed 
for solving the split feasibility problem.
Its convergence is proved. Under a general condition,
we prove that: if the SFP is feasible,
we are able to compute a feasible solution;
if it is infeasible, we can obtain a certificate for the infeasibility.

\bigskip
{\bf Acknowledgment.} \,
Jiawang Nie was partially supported by the NSF grants
DMS-1417985 and DMS-1619973.
Jinling Zhag was partially supported by the National Natural Science Foundation of China,
under the grants 11101028, 11271206,
and the Fundamental Research Funds for the Central Universities.
This work is conducted when Jinling Zhao is a
visiting scholar at the University of California, San Diego.


\end{document}